\documentclass[12pt,reqno]{amsart}
\usepackage{hyperref}
\usepackage{amsmath,amssymb, amsthm}
\usepackage{caption}
\usepackage{amssymb,amsmath,euscript, enumerate,tikz}
\usepackage[margin=1in]{geometry}
\usepackage{graphicx}
\usepackage{float}
\usepackage{pgf,tikz}
\usepackage{mathrsfs}
\usepackage{subfigure}
\usepackage{mathrsfs}
\usepackage{mathtools}

\newcommand \iv{\operatorname{iv}}

\newcommand \cdeg{\operatorname{cdeg}}

\newcommand \hs{\operatorname{HS}}
\newcommand \hf{\operatorname{HF}}
\newcommand \hp{\operatorname{HP}}

\newcommand \h{\operatorname{ht}}

\newcommand \K{\mathbb{K}}

\DeclareMathOperator{\fcdeg}{fcdeg}
\DeclareMathOperator{\height}{ht}

\theoremstyle{plain}
\newtheorem{theorem}{Theorem}[section]
\newtheorem{lemma}[theorem]{Lemma}

\newtheorem{proposition}[theorem]{Proposition}
\newtheorem{corollary}[theorem]{Corollary}

\theoremstyle{definition}

\newtheorem{remark}[theorem]{Remark}
\newtheorem{definition}[theorem]{Definition}
\newtheorem{example}[theorem]{Example}

\begin{document}

\title[On the Hilbert series and multiplicity of Generalized Binomial Edge Ideals]{On the Hilbert series and multiplicity of Generalized Binomial Edge Ideals}

 \author[Paramhans Kushwaha]{Paramhans Kushwaha}
 \email{2022rma2004@iitjammu.ac.in\\paramhans844@gmail.com}
 \address{Indian Institute of Technology, Jammu}
\address{NH-44, PO Nagrota, Jagti, Jammu and Kashmir 181221}

	\begin{abstract}
       In this article, we compute the Hilbert series of the generalized binomial edge ideal associated with the join of two graphs. Consequently, we compute the Hilbert series of complete \(r\)-partite graphs. Further, we show that if a vertex satisfies a certain degree condition, then some Hilbert coefficients remain unchanged upon its removal. This provides a reduction technique for computing them. As an application, we compute the dimension and multiplicity of several classes of graphs. 
	\end{abstract}
	
	\maketitle
  \section{introduction}
  Let $X=(x_{ij})_{m\times n}$ be a matrix of variables, and $S=\K[x_{ij}:(i,j)\in [m]\times [n]]$ be a polynomial ring over a field $\K$. The ideal generated by all $t$-minors of $X$, for $t\leq \min\{m,n\}$, is called the \emph{determinantal ideal} and is denoted by $I_t(X)$. These ideals are extensively studied in the literature due to their significance in combinatorics, representation theory, and invariant theory (see \cite{Grobner-Bases-and-Determinantal-ideals-Bruns-Conca}, \cite{Determinantal-rings-Bruns-Vetter}). Subsequently, researchers shifted their attention to ideals generated by arbitrary sets of minors of a generic matrix, as these ideals arise naturally in many branches of mathematics.

The binomial edge ideal of a pair of graphs, introduced in \cite{The-binomial-edge-ideal-of-pair-of-graphs-Ene-Herzog-Hibi-Qureshi}, is a generalization of the determinantal ideal $I_2(X)$ of an $m\times n$ generic matrix $X$. Let $G_1$ and $G_2$ be simple graphs on vertex sets $[m]$ and $[n]$, respectively. Let $\K[X]$ be  the polynomial ring over a field $\K$, where $X=(x_{ij})_{m\times n}$ is a matrix of variables. For each pair $(e,f)\in E(G_1)\times E(G_2)$ such that $e=\{i,j\}$ and $f=\{k,l\}$ with $i<j$ and $k<l$, assign a $2$-minor $p_{(e,f)}=[i~j|k~l]=x_{ik}x_{jl}-x_{il}x_{jk}$. The \emph{binomial edge ideal of the pair $(G_1,G_2)$}, denoted by $J_{G_1,G_2}$, is an ideal in $S$ defined as follows:
$$J_{G_1,G_2}=\langle p_{(e,f)}:e\in E(G_1),f\in E(G_2)\rangle.$$
This framework simultaneously generalizes the ideal of adjacent minors and generalized binomial edge ideals. If $(G_1,G_2)=(K_m,G)$, where $K_m$ is a complete graph on $m$ vertices, then $J_{K_m,G}$ or simply $J_{m,G}$ is the \emph{generalized binomial edge ideal} associated to $G$, introduced in \cite{GBI-Rauh-2013}. For $m=2$, the ideal $J_{K_2,G}$ is the classical binomial edge ideal $J_G$ of $G$ appeared independently in \cite{BEIandConditionalDependance,GraphsandIdealsGeneratedbysome2-minor-Ohtani-Masahiro}. The binomial edge ideal and the generalized binomial edge ideal exhibit fundamentally different behavior when \(m\geq 3\). For instance, \(J_{m,G}\) is Cohen--Macaulay if and only if \(G\) is a complete graph \cite{CM-GBEI-Amata-Crupi-Rinaldo}. In contrast, Cohen--Macaulay binomial edge ideals are not restricted to complete graphs. This distinction highlights the need for a separate study of generalized binomial edge ideals.

The binomial and generalized binomial edge ideals are important in algebraic statistics, particularly in the study of conditional independence ideals. Researchers have been trying to study the algebraic properties of these ideals in terms of the combinatorial properties of the underlying graph(s). For example, the associated primes of binomial and generalized binomial edge ideals admit a combinatorial description via the cut point property of the underlying graph. Significant advancements have been achieved so far in the study of binomial edge ideals, see (\cite{CM-BEI-Ene-H-H,OntheextremalBettinumofBEIofblockgraphds,Partial-Betti-splittings-applications-to-BEI-Jayanthan-Shivakumar-Van-tuyl,Some-CM-and-unmixed-BEI-Kiani-Madani-2015,BEI-of--generalized-block-graphs-Arvind-Kumar,HS-of-BEI-Arvind-Sarkar-2019,KK25,Krull-dim-and-reg-BEI-of-Block-graphs-Rinaldo,A-proof-for-a-conjecture-reg-BEI-Malayeri-Madani-Kiani}) and references therein. The depth of generalized binomial edge ideals is studied in limited works \cite{On-the-Depth-of-GBEI-Anuvinda-Mehta-Saha,shen-Zhu-2023-GBEI-of-complete-r-partite-graphs}. The Castelnuovo-Mumford regularity for binomial and generalized binomial edge ideals has been extensively studied. We refer the reader to \cite{On-a-reg-conjecture-of-GBEI-Anuvinda-Mehta-Saha,Castelnuovo-Mumford-reg-of-GBEI-of-graphs-Kiani-Madani-Zhu} and a recent survey \cite{jayanthan-Kumar-2025generalizedbinomialedgeideals}, and references therein. Despite the growing interest in generalized binomial edge ideals (\cite{Arithmetical-rank-and-cohomological-dimension-of-GBEI-Katsabekis,Powers-of-GBEI-of-path-graphs-shen-Zhu,GBEI-of-bipartite-graphs}), their Hilbert series remains largely unexplored.

The objective of this paper is to find the Hilbert series of the generalized binomial edge ideals of graphs using the Hilbert series of certain subgraphs. Let $H$ and $H'$ be two graphs with the vertex set $[p]$ and $[q]$, respectively. The join of $H$ and $H'$, denoted by $H*H'$ is the graph with vertex set $[p]\sqcup [q]$ and the edge set $E(H*H')=E(H)\cup E(H')\cup \{\{i,j\}:(i,j)\in [p]\times[q]\}$. Our goal is to compute the Hilbert series of the join of $H$ and $H'$ in terms of the Hilbert series of $H$ and $H'$ as follows:

 \medskip
    \noindent
    \textbf{Theorem~~\ref{thm: hs of join of graphs}.} 
 Let $H$ and $H'$ be two graphs on vertex sets $[p]$ and $[q]$, respectively. Let $G=H*H'$ be the join of $H$ and $H'$. Then 
    \begin{align*}
        \hs(S/J_{m,H*H'},t)=&\hs(S_H/J_{m,H},t)+\hs(S_{H'}/J_{m,H'},t)\\
        & +\frac{\sum_{i\geq 0}\binom{m-1}{i}\left[\binom{p+q-1}{i}-\binom{p-1}{i}(1-t)^q-\binom{q-1}{i}(1-t)^p\right]t^i}{(1-t)^{m+p+q-1}}.
    \end{align*}

The Hilbert polynomial is a fundamental tool in the study of finitely generated graded algebras, as it provides valuable information about their algebraic structure. The coefficients of the Hilbert polynomial, when written in the standard binomial basis, are called the Hilbert coefficients. Among these, the leading coefficient $e_0$, known as the multiplicity, is one of the most significant numerical invariants because of its close connection with the geometry of algebraic varieties. Although multiplicity has been widely studied, its computation remains difficult for general classes of ideals, and the determination of the higher Hilbert coefficients is even more challenging. While binomial edge ideals have been investigated extensively, their generalization, namely generalized binomial edge ideals, has received comparatively less attention. At present, the study of invariants such as the Hilbert series and multiplicity for generalized binomial edge ideals is still in its early stages.

We provide certain graph-theoretic conditions on a vertex under which Hilbert coefficients remain unchanged upon its removal. In fact, we prove:

\medskip
    \noindent
    \textbf{Theorem~~\ref{thm: fcdeg condition to remove vertex keeping Hilbert coefficients same}.} 
 Let $G$ a graph and $v\in \iv(G)$ such that $\fcdeg(v)\geq \lceil\frac{i+2}{m-1}\rceil+1$ for some $i\geq 0$. Then $$e_j(S/J_{m,G})=e_j(S/J_{m,G\setminus v}),$$ for all $0\leq j\leq i.$

As a consequence, we compute the dimension and multiplicity of several classes of graphs. 

The article is organized as follows: In Section \ref{sec: Preliminaries}, we recall the necessary algebraic and graph-theoretic concepts relevant to this work. Section~\ref{sec: HS of join of graphs} is devoted to computing the Hilbert series of the join of two graphs in terms of the Hilbert series of the individual graphs. 
In Section \ref{sec: multiplicity}, we show that some of the Hilbert coefficients remain unchanged after removing certain vertices. Section \ref{sec: Applications} is devoted to computing the dimension and multiplicity of several classes of graphs.
\section{Preliminaries}\label{sec: Preliminaries}
In this section, we recall basic definitions and results from commutative algebra and combinatorics that will be used throughout the paper.
\subsection{Basics on Graph Theory}
A graph $G$ is an ordered pair $(V(G),E(G))$, where $V(G)$ is the vertex set of $G$ and $E(G)$ is the edge set of $G$, consisting of $2$-element subsets of $V(G)$. A graph $G$ is called \emph{simple} if it has no multiple edges or loops. A graph $H$ is called a \emph{subgraph} of $G$ if $V(H)\subset V(G)$ and $E(H)\subset E(G)$. A subgraph $H$ is an \emph{induced subgraph} of $G$ if whenever $\{i,j\}\in E(G)$ with $i,j\in V(H)$ implies $\{i,j\}\in E(H)$. A graph $G$ is called a \emph{clique} or a \emph{complete} graph if there is an edge between every distinct pair of vertices of $G$. A complete graph on $m$ vertices is denoted by $K_m$. We say an induced subgraph $H$ is a \emph{clique of $G$} if $H$ is a complete graph. A \emph{maximal clique} of $G$ is a clique of $G$ that is not contained in any other clique of $G$. For $v\in V(G)$, denote $\cdeg_G(v)$ the number of maximal cliques of $G$ containing $v$, called the \emph{clique degree} of $v$. A vertex of clique degree one is called a \emph{free vertex}; otherwise, it is called an \emph{internal vertex}. The set of all internal vertices of $G$ is denoted by $\iv(G)$. For $v\in V(G)$, denote $N_G(v)$ the collection of adjacent vertices to $v$ in $G$, and $N_G[v]=N_G(v)\cup \{v\}$. Denote $G_v$ the graph on $V(G)$ obtained by adding necessary edges so that the induced subgraph on $N[v]$ is complete. The graph $G\setminus v$ is a graph on the vertex set $V(G)\setminus\{v\}$ with $E(G\setminus v)=E(G)\setminus{\{\{u,v\}: u\in N(v)\}}$. For any graph-theoretic terminology not explained here, we refer the reader to look at the standard text of graph theory \cite{BondyMurty2008}.
\subsection{Algebraic Basics}
Let \(R=\mathbb{K}[z_1,\ldots,z_n]\) be a polynomial ring over a field \(\mathbb{K}\), where \(\deg(z_i)=1\) for all \(1\le i\le n\), and let \(I\subseteq R\) be a homogeneous ideal. We make a convention that $\mathbb {N} $ contains $0$. The \emph{Hilbert function} of the graded ring $R/I$ is the function $\hf(R/I, \rule{0.2cm}{0.15mm}):\mathbb{N}\longrightarrow \mathbb{N}$ with 
    \[
    \hf(R/I, j)=\dim_{\K}(R/I)_j
    \]
    for all $j\in \mathbb N$, where $(R/I)_j$ denotes the $j$-th graded component of $R/I$. The generating function of this numerical function is called the \emph{Hilbert Series} of $R/I$, and is given as follows
    \[
    \hs(R/I,t)=\sum_{j\in \mathbb{N}}\hf(R/I,j)\ t^j.
    \]
    The Hilbert function is additive on short exact sequences, so is the Hilbert series. A classical theorem of Hilbert-Serre provides a reduced form of the Hilbert series
    \[
    \hs(R/I,t)=\frac{Q(t)}{(1-t)^d},
    \]
    where $Q(t)\in \mathbb{Z}[t]$ is a polynomial with $Q(1)\geq 1$ and $d=\dim(R/I)$. The Hilbert polynomial of $R/I$, denoted by $\hp(R/I, X)$ is the unique polynomial with rational coefficients such that $\hf(R/I,j)=\hp(R/I, j)$ for $j\gg 0$.
    Moreover, the coefficients of the Hilbert polynomial can be recovered from the numerator $Q(t)$. More precisely, if $\hp(R/I, X)$ denotes the Hilbert polynomial of $R/I$, and is written in the form
    \[
    \hp(R/I, X)= \sum_{i=1}^{d-1}(-1)^{d-1-i}e_{d-1-i} \left( \begin{array}{c} X+i \\ i \end{array} \right),
    \]
    then the coefficients $e_i \text{ for }0\leq i\leq d-1$ can be obtained as  follows: $e_i=\frac{Q^{(i)}(1)}{i!}$, where $Q^{(i)}(t)$ denotes the $i$-th derivative of the polynomial $Q(t)$. The number $e_i$ is called the 
    \emph{$i$-th Hilbert coefficient} of $R/I$, and is denoted by $e_i(R/I)$ for $0\leq i\leq d-1$. In particular, when $\dim(R/I)>0$, the coefficient $e_0(R/I)$ is the well-known \emph{multiplicity} of $R/I$. 

Let $I$ be an ideal in the polynomial ring $S$, and $I=\cap_j I_j$ be the irredundant primary decomposition of $I$. The following shows that the multiplicity of $S/I$ is the sum of multiplicities of $S/I_i$ for which $\dim(S/I)=\dim(S/I_i)$. 
\begin{lemma}\label{lemma: multiplicity of arbitrary ideal}
    Let $I=(\cap_{i=1}^rI_i)\cap(\cap_{j=1}^sJ_j)$ be the irredundant primary decomposition, where $\dim(S/I)=\dim(S/I_i)$ for all $1\leq i\leq r$ and $\dim(S/I)>\dim(S/J_j)$ for all $1\leq j\leq s$. Then $$e_0(S/I)=\sum_{i=1}^r e_0(S/I_i).$$ 
\end{lemma}
\begin{proof}
Let $I=(\cap_{i=1}^rI_i)\cap J$, where $J=\cap_{j=1}^sJ_j$. Under the given hypothesis, we have $\h(J)>\h(\cap_{i=1}^rI_i)$. Therefore, it follows from \cite[Lemma 3.1]{KKP-Hilbert-coefficients-regularity-BEI} that $e_0(S/I)=e_0(S/{\cap_{i=1}^r I_i})$. Therefore, we assume that $I=\cap_{i=1}^rI_i$, where $I_i$ is $P_i$-primary, $P_i\neq P_j$ for $i\neq j$ and $\dim(S/I)=\dim(S/I_i)$ for all $1\leq i\leq r$. 

We prove the statement using induction on $r$. For $r=1$, there is nothing to prove. Now, assume $r\geq2$, and let $I=L\cap I_r$, where $L=\cap_{i=1}^{r-1}I_i$. Note that $\height(L+I_r)\geq \height(I_r)$.  We claim that $\height(L+I_r)>\height(I_r).$ Suppose not, and $\height(L+I_r)=\height(I_r)=h$. Let $Q$ be a minimal prime of $L+I_r$ such that $\height(Q)=h$. Then $I_r\subset Q$, so $P_r\subset Q$. Since $\height(P_r)=\height(Q)$, we get $P_r=Q$. Also $L=\cap_{i=1}^{r-1}I_i\subset Q$, so there exists $i<r$ such that $I_i\subset Q$. Then it follows that $P_i\subset Q$, and therefore $P_i=Q$. This contradicts the fact that $P_r\neq P_i$ for $i<r$. Hence $\height(L+I_r)>\height(I_r).$ Now, consider the short exact sequence \[ 0\rightarrow{\frac{S}{I}}\rightarrow \frac{S}{L}\oplus \frac{S}{I_r}\rightarrow \frac{S}{L+I_r}\rightarrow 0.\] Then, we have $\hs(S/I,t)=\hs(S/L,t)+\hs(S/I_r,t)-\hs(S/L+I_r,t)$. Observe that $\height(L)=\height(I_r)$. Therefore, we obtain $e_0(S/I)=e_0(S/L)+e_0(S/I_r)$. Using induction, we get the desired result.
\end{proof}

We now recall several results concerning the main algebraic object of this article, namely, generalized binomial edge ideals. These ideals have attracted considerable attention from researchers in recent years. We collect the properties required in the subsequent sections. Let $G$ be a simple graph on $[n]$ and $J_{m,G}$ be the generalized binomial edge ideal of $G$ in the polynomial ring $S=\K[x_{ij}:i\in [m],j\in [n]]$. The associated primes of $J_G$ are first determined in \cite{GBI-Rauh-2013}, and they are given using the connectivity property of the graph $G$. Let $T\subseteq [n]$, and consider $G\setminus T$, the induced subgraph of $G$ on the vertex set $[n]\setminus T$. Suppose that $G_1,\ldots , G_{c_G(T)}$ are connected components of $G\setminus T$. For $1\leq i\leq c_G(T)$, let $\widetilde{G_i}$ denote the complete graph on the vertex set $V(G_i)$. Now consider the ideal 
    \[
    P_T(m,G)\coloneqq \left \langle \{x_{ij}:(i,j)\in [m]\times T\} , J_{m,\widetilde{G_1}}, \ldots ,  J_{m,\widetilde{G_{c(T)}}}  \right\rangle.
    \]
    The ideal $P_T(G)$ is a prime ideal of $S$ containing $J_{m,G}$. By \cite[Corollary 4]{GBI-Rauh-2013}, the ideal $J_{m,G}$ is radical. To identify the minimal prime ideals among $P_T(G)$, we require a graph-theoretic notion known as the cut point property. A vertex $v\in V(G)$ is called a \emph{cut vertex} if the number of components of $G$ is less than that of $G\setminus v$. A subset $T\subset [n]$ is said to have the \emph{cut point property} or $T$ is said to be a \emph{cut set} of $G$ if $v$ is a cut vertex of $G\setminus(T\setminus v)$ for each $v\in T$. Denote $\mathscr C(G)$ the collection of all subsets of $V(G)$ having the cut point property, including the empty set. As shown in \cite[Theorem 7]{GBI-Rauh-2013}, the minimal primes of $J_{m,G}$ are precisely the prime ideals $P_T(m,G)$ for $T\in \mathscr{C}(G)$, i.e., $$J_{m, G}=\bigcap_{T\in \mathscr C(G)}P_T(m,G).$$ The height of the minimal prime $P_T(m,G)$ is $$\height(P_{T}(m,G))=(m-1)(n-c(T))+|T|.$$
So, the dimension is $$\dim(S/J_{m,G})=\max\left\{(m-1)c(T)+n-|T|\right\}.$$

We now collect some results that are used throughout this article. We begin with a result of Kumar \cite{Reg-bounds-of-GBEI-of-graphs-Arvind-2020}, which provides a useful short exact sequence.
\begin{theorem}\label{thm: Short exact sequence}\cite[Theorem 3.2]{Reg-bounds-of-GBEI-of-graphs-Arvind-2020}
    If $v\in  V(G)$ is a internal vertex, then $$J_{m,G}=J_{m,{G_v}}\cap \langle\{ x_{iv}:i\in [m]\}+J_{m,{G\setminus v}}\rangle.$$
\end{theorem}
The following two results provide the cut sets of the join of graphs. 
\begin{proposition}\label{prop: cut sets of join of two disconnected graphs}\cite[Proposition 4.14]{Some-CM-and-unmixed-BEI-Kiani-Madani-2015}
    Let $H$ and $H'$ be disconnected graphs on $[p]$ and $[q]$, respectively. Then $$\mathscr C(H*H')=\emptyset \bigcup\{T_1\sqcup [q]:T_1\in \mathscr C(H)\}\bigcup \left\{[p]\sqcup T_2:T_2\in \mathscr C(H'\right\}.$$
\end{proposition}
\begin{lemma}\label{lemma: cut sets of join of a disconnected graph and a complete graph}\cite[Lemma 4.3]{HS-of-BEI-Arvind-Sarkar-2019}.
    Let $H$ be a disconnected graph on $[p]$. Let $G=H*K_q$ be the join of $H$ and the complete graph $K_q$. Then $$\mathscr C(G)=\emptyset\cup \left\{T\sqcup[q]:T\in \mathscr C(H)\right\}.$$
\end{lemma}
We need the Hilbert series of the generalized binomial edge ideal $J_{m,K_n}$ of $K_n$. Observe that $J_{m,K_n}$ is the classical determinantal ideal $I_2(X)$ of $2$-minors of an $m\times n$ generic matrix. Therefore, we have the following:  
\begin{corollary}\cite[Corollary 1]{on-the-Hilbert-series-of-determinantal-rings-and-theire-canonical-module-Conca-Herzog-1994}\label{cor: hs of 2-minors} The Hilbert series of $S/J_{m,K_n}$ is given by
    \[
\hs(S/J_{m,K_n})
=
\frac{
\displaystyle
\sum_{i\geq0}
\binom{m-1}{i}
\binom{n-1}{i}
t^i
}
{(1-t)^{m+n-1}}.
\]
\end{corollary}
\section{Hilbert Series of Join of Graphs}\label{sec: HS of join of graphs}
This section is devoted to computing the Hilbert series of the generalized binomial edge ideal of the join of two graphs.

\begin{definition}\label{def: join of graphs}
Let $H$ and $H'$ be two graphs on the vertex set $[p]$ and $[q]$, respectively. The \emph{join} of $H$ and $H'$, denoted by $H*H'$ is the graph with vertex set $[p]\sqcup[q]$ and the edge set $E(H)\cup E(H')\cup \{\{i,j\}|i\in[p],j\in[q]\}$. For example, for any $n\geq 2$, $K_n$ is a join of $K_{n_1}$ and $K_{n_2}$ with $n_1+n_2=n$, i.e., $K_n=K_{n_1}*K_{n_2}$ such that $n_1+n_2=n$. The graph in Figure \ref{fig: join of graphs} is the join of $H=\{\{1\},\{2\},\{3\}\}$ and $H'=\{\{4\},\{5,6\}\}$.
\end{definition}
\begin{figure}[H]
\centering
\begin{tikzpicture}[scale=0.5, line cap=round,line join=round,,x=1cm,y=1cm]
\draw  (0,0)-- (4,2);
\draw  (0,-2)-- (4,2);
\draw  (0,2)-- (4,2);
\draw  (0,0)-- (4,0);
\draw  (0,-2)-- (4,0);
\draw  (0,2)-- (4,0);
\draw  (0,0)-- (4,-2);
\draw  (0,-2)-- (4,-2);
\draw  (0,2)-- (4,-2);
\draw  (4,2)-- (4,0);
\fill (0,0) circle (3.5pt);
\fill (4,0) circle (3.5pt);
\begin{scriptsize}
    \fill (0,0) circle (3.5pt);
\fill (0,2) circle (3.5pt);
\fill (0,-2) circle (3.5pt);
\fill (4,0) circle (3.5pt);
\fill (4,2) circle (3.5pt);
\fill (4,-2) circle (3.5pt);
\end{scriptsize}
 \draw[color=black] (-0.35,-0.35) node {$2$};
\draw[color=black] (-0.35,2) node {$1$};
\draw[color=black] (-0.35,-2) node {$3$};
\draw[color=black] (4.35,-2) node {$4$};
\draw[color=black] (4.35,0) node {$5$};
\draw[color=black] (4.35,2) node {$6$};
\end{tikzpicture}
\caption{} \label{fig: join of graphs}
\end{figure}

We set up some notation that we use throughout the section. Let $H$ and $H'$ be two graphs on vertex sets $[p]$ and $[q]$, respectively. Let $G=H*H'$ be the join of $H$ and $H'$. Set $S_H=\K[x_{ij}:i\in [m], j\in V(H)]$, $S_{H'}=\K[y_{ik}:i\in[m], k\in V(H')]$ and $S=\K[x_{ij},y_{ik}:i\in [m], j\in V(H),k\in V(H')]$. Further, let $S_l$ denote the polynomial ring of $J_{m,K_{l}}$ for any $l\in \mathbb N$. 

We first consider the disconnected case and give the Hilbert series of the join of two disconnected graphs as follows:
\begin{theorem}\label{thm: hs of join of two disconnected graphs}
    Let $H$ and $H'$ be two disconnected graphs on vertex sets $[p]$ and $[q]$, respectively. Let $G=H*H'$ be the join of $H$ and $H'$. Then 
    \begin{align*}
        \hs(S/J_{m,H*H'},t)=&\hs(S_H/J_{m,H},t)+\hs(S_{H'}/J_{m,H'},t)\\
        & +\frac{\sum_{i\geq 0}\binom{m-1}{i}\left[\binom{p+q-1}{i}-\binom{p-1}{i}(1-t)^q-\binom{q-1}{i}(1-t)^p\right]t^i}{(1-t)^{m+p+q-1}}.
    \end{align*}
\end{theorem}
\begin{proof}
    It follows from Proposition \ref{prop: cut sets of join of two disconnected graphs} that 
    $$\mathscr C(H*H')=\emptyset \bigcup\{T_1\sqcup [q]:T_1\in \mathscr C(H)\}\bigcup \left\{[p]\sqcup T_2:T_2\in \mathscr C(H'\right\}.$$
    Let \[
Q_1=\bigcap_{\substack{T\in\mathscr{C}(G)\\ [p]\subseteq T}}P_T(m,G)=(x_{ij}:i\in[m],j\in [p])+J_{H'} 
\] and  \[
Q_2=\bigcap_{\substack{T\in\mathscr{C}(G)\\ T\neq\emptyset,[p]\not\subseteq T}}P_T(m,G)=(y_{ik}:i\in[m],k\in [q])+J_{H}. 
\]
Then, by \cite[Theorem 7]{GBI-Rauh-2013}, we have 
\[J_{m,G}=\bigcap_{T\in \mathscr{C}(G)}P_T(m,G)=Q_1\cap Q_2\cap P_{\emptyset}(m,G)=Q_1\cap Q_2\cap J_{m,K_{p+q}}.\]
Set $Q_3=Q_2\cap  J_{m,K_{p+q}}$. Then we have the short exact sequence 
\[ 0\rightarrow{\frac{S}{Q_3}}\rightarrow \frac{S}{Q_2}\oplus \frac{S}{J_{m,K_{p+q}}}\rightarrow \frac{S}{Q_2+J_{m,K_{p+q}}}\rightarrow 0.\]
Observe that $Q_2+J_{m,K_{p+q}}=(y_{ik}:i\in[m],k\in [q])+J_{m,K_p}.$ Therefore, we get
\begin{align*}
    \hs(S/{Q_3},t)&=\hs(S/{J_{m,K_p+q}},t)+\hs(S/{Q_2},t)-\hs(S/({Q_2+J_{m,K_{p+q}}}),t)\\
    &=\hs(S/{J_{m,K_p+q}},t)+\hs(S_H/{J_{m,H}},t)-\hs(S_p/{J_{m,K_{p}}},t).
\end{align*}
Now, consider the short exact sequence 
\[ 0\rightarrow{\frac{S}{J_{m,G}}}\rightarrow \frac{S}{Q_1}\oplus \frac{S}{Q_3}\rightarrow \frac{S}{Q_1+Q_3}\rightarrow 0.\]
Notice that $Q_1+Q_3=(x_{ij}:i\in[m],j\in [p])+J_{m,K_q}$. Therefore, we have
\begin{align*}
    \hs(S/J_{m,G},t)&=\hs(S/{Q_1},t)+\hs(S/{Q_3},t)-\hs(S/(Q_1+Q_3),t)\\
    &=\hs(S_{H'}/{J_{m,H'}},t)+\hs(S_H/{J_{m,H}},t)+\hs(S/{J_{m,K_p+q}},t)\\&\quad -\hs(S_p/{J_{m,K_{p}}},t)-\hs(S_q/{J_{m,K_{q}}},t).
\end{align*}
Now, the result follows from Corollary \ref{cor: hs of 2-minors}.
\end{proof}
We require the Hilbert series of the product of two graphs defined in \cite{HS-of-BEI-Arvind-Sarkar-2019} to compute the Hilbert series of arbitrary graphs. 
Let $H$ be a graph on $[p]$ and $H'$ be a graph on $[q]$. Let $v_1,\ldots,v_r$ be vertices of $H'$ with $1\leq r\leq q$. Denote $H(*)^rH'$ the graph on vertex set $[p]\sqcup[q]$ and edge set $$E(H)\sqcup E(H')\sqcup\left\{\{u,v_i\}:u\in V(H), i=1,\ldots,r\right\}.$$ If $r=q$, then $H(*)^qH'=H*H'$ is the join of $H$ and $H'$. The graph shown in Figure \ref{fig: join of H and some points of K_q} is $H(*)^2K_4$, where $H=\{\{6,7\},\{7\}\}$.
\vspace{2mm}
\begin{figure}[h]
    \centering
     \begin{tikzpicture}[scale=0.6,line cap=round,line join=round,
    x=1cm,y=1cm]
\draw (0,0)--(0,2);
\draw (0,2)--(2,2);
\draw (2,0)--(2,2);
\draw (2,0)--(0,0);
\draw (0,0)--(2,2);
\draw (2,0)--(0,2);
\draw (-2,-2)--(0,0);
\draw (-2,-2)--(0,2);
\draw (-2,0)--(-2,2);
\draw (-2,2)--(0,2);
\draw (-2,2)--(0,0);
\draw (-2,0)--(0,2);
\draw (-2,0)--(0,0);
\foreach \x/\y in {
    0/0,
    0/2,
    2/2,
    2/0,
    -2/-2,
    -2/0,
    -2/2
}
{
    \fill (\x,\y) circle (2.5pt);
}
\node[below right] at (0,0) {$1$};
\node[above right] at (0,2) {$2$};
\node[above right] at (2,2) {$3$};
\node[below right] at (2,0) {$4$};

\node[below left] at (-2,-2) {$5$};
\node[below left] at (-2,0) {$6$};
\node[above left] at (-2,2) {$7$};
\end{tikzpicture}
    \caption{}\label{fig: join of H and some points of K_q}
    \label{fig:placeholder}
\end{figure}
\newline
Now, we give the Hilbert series of $H(*)^rK_q$ as follows.
\begin{lemma}\label{lemma: hs of join of a graph and some points of a complete graph}
    Let $H$ be a graph on $[p]$ and $G=H(*)^rK_q$ be the graph on the vertex set $[p]\sqcup[q]$ with $q\geq2$ and $1\leq r<q$. Then \begin{align*}
    \hs(S/J_{m,G},t)&=\hs(S_H/J_{m,H},t)\hs(S_{{q-r}}/J_{m,K_{q-r}},t)+\hs(S/{J_{m,K_p+q}},t)\\
    &\quad\quad-\hs(S_{{p+q-r}}/J_{m,K_{p+q-r}},t).
    \end{align*}
\end{lemma}
\begin{proof}
   Let $v_1,\ldots,v_r$ be the vertices of $K_q$. Then, by \cite[Lemma 4.8]{HS-of-BEI-Arvind-Sarkar-2019}, we have \[\mathscr{C}(G)=\{\emptyset\}\cup\left\{T\sqcup\{v_1,\ldots,v_r\}:T\in \mathscr{C}(H)\right\}.\]
  It follows from \cite[Theorem 7]{GBI-Rauh-2013} that
 \begin{align*}
 J_{m,G}&=\bigcap_{T\in \mathscr{C}(G)}P_T(m,G)=P_{\emptyset}(m,G)\cap\bigcap_{\substack{T\in \mathscr{C}(G),T\neq \emptyset}}P_T(m,G)\\
 &=J_{m,K_{p+q}}\cap\left((x_{iv_j}:i\in [m], j\in[r])+J_{m,H}+J_{m,K_{q-r}}\right).\end{align*}
 Let $Q=(x_{iv_j}:i\in [m], j\in[r])+J_{m,H}+J_{m,K_{q-r}}$. Note that $$J_{m,K_{p+q}}+Q=(x_{iv_j}:i\in [m], j\in[r])+J_{m,K_{p+q-r}}.$$ Then the short exact sequence below 
 \[ 0\rightarrow{\frac{S}{J_{m,G}}}\rightarrow \frac{S}{J_{m,K_{p+q}}}\oplus \frac{S}{Q}\rightarrow \frac{S}{Q+J_{m,K_{p+q}}}\rightarrow 0\] gives
 \begin{align*}  
 \hs(S/J_{m,G},t)&=\hs(S/Q,t)+\hs(S/J_{m,K_{p+q}},t)-\hs(S/(Q+J_{m,K_p+q}),t)\\
    &=\hs(S_H/J_{m,H},t)\hs(S_{{q-r}}/J_{m,K_{q-r}},t)+\hs(S/{J_{m,K_p+q}},t)\\
    &\quad\quad-\hs(S_{{p+q-r}}/J_{m,K_{p+q-r}},t).
 \end{align*} This completes the proof.
\end{proof}
We now turn to the connected case, beginning with \(H*K_q\), where \(H\) is connected. Some additional results are required for this purpose. We first consider the case \(q=1\).
\begin{lemma}\label{lemma: HS of connected and a vertex}
    Let $H$ be a connected graph on $[p]$ and $G=H*\{v\}$. Then 
    $$\hs(S/J_{m,G},t)=\hs(S_H/J_{m,H},t)+\hs(S_{p+1}/J_{m,K_{p+1}},t)-\hs(S_p/J_{m,K_p},t).$$
    \begin{proof}
        The result trivially follows if $H=K_p$. Now, assume that $H\neq K_p$, so that $v$ is an internal vertex in $G$. Then, by Theorem \ref{thm: Short exact sequence}, we have $$J_{m,G}=J_{m,{G_v}}\cap \langle\{ x_{iv}:i\in [m]\}+J_{m,{G\setminus v}}\rangle.$$ Note that $G_v=K_{p+1}$ and $G\setminus v=H$. The short exact sequence
        \[ 0\rightarrow{\frac{S}{J_{m,G}}}\rightarrow \frac{S}{J_{m,K_{p+1}}}\oplus \frac{S}{(x_{iv}:i\in [m])+J_{m,H}}\rightarrow \frac{S}{(x_{iv}:i\in [m])+J_{m,K_{p}}}\rightarrow 0\] then implies that 
         $$\hs(S/J_{m,G},t)=\hs(S_H/J_{m,H},t)+\hs(S_{p+1}/J_{m,K_{p+1}},t)-\hs(S_p/J_{m,K_p},t).$$ Thus, the statement follows.
    \end{proof}
\end{lemma}
 We now compute the Hilbert series of a $q$-cone over a connected graph $H$. Denote the $q$ isolated vertices by $K_q^c$.
 \begin{lemma}\label{lemma: hs of join of connected graph and isolated vertices}
     Let $H$ be a connected graph on $[p]$. Let $G=H*K_q^c$ be the join $H$ and $K_q^c$. Then $$\hs(S/J_{m,G},t)=\hs(S_H/J_{m,H},t)+\hs(S/J_{m,K_p*K_q^c},t)-\hs(S_p/J_{m,K_p},t).$$
 \end{lemma}
 \begin{proof}
 The result is immediate if $H=K_q$. Therefore, we assume that $H$ is not a complete graph.
     We proceed by induction on $q$. For $q=1$, the graph $G$ is $H*\{v\}$ and the result follows from Lemma \ref{lemma: HS of connected and a vertex}. Now assume $q>1$ and that the statement holds for $q-1$. Let $G=H*K_q^c$ and $V(K_q^c)=\{v_1,\ldots,v_q\}$. Set $Q_1=(x_{iv_q}:i\in [m])+J_{m,G\setminus {v_q}}$ and $Q_2=J_{m,G_{v_q}}$. Note that $Q_1+Q_2=(x_{iv_q}:i\in [m])+J_{m,G_{v_q}\setminus {v_q}}$. Since $v_q$ is not a free vertex, it follows from Theorem \ref{thm: Short exact sequence} that $J_{m,G}=Q_1\cap Q_2$. Let $R=\K[y_{ik}:k\in V(G)\setminus\{v_q\}]$. Notice that $G\setminus v_q=H*K_{q-1}^c$. Then, by the induction hypothesis, we get  $$\hs(S/Q_1,t)=\hs(S_H/J_{m,H},t)+\hs(R/J_{m,K_p*K_{q-1}^c},t)-\hs(S_p/J_{m,K_p},t).$$
     Consider the short exact sequence 
     \[ 0\rightarrow{\frac{S}{J_{m,G}}}\rightarrow \frac{S}{Q_1}\oplus \frac{S}{Q_2}\rightarrow \frac{S}{Q_1+Q_2}\rightarrow 0.\]
Note that $G_{v_q}=K_p*K_q^c$, $G_{v_q}\setminus v_q=K_p*K_{q-1}^c$. Therefore, it follows from the above short exact sequence that 
\begin{align*}
    \hs(S/J_{m,G},t)&=\hs(S/Q_1,t)+\hs(S/Q_2,t)-\hs(S/Q_1+Q_2,t)\\
    &=\hs(S_H/J_{m,H},t)+\hs(R/J_{m,K_p*K_{q-1}^c},t)-\hs(S_p/J_{m,K_p},t)\\
    &\quad +\hs(S/J_{m,K_p*K_q^c},t)-\hs(R/J_{m,K_p*K_{q-1}^c},t)\\
    &=\hs(S_H/J_{m,H},t)+\hs(S/J_{m,K_p*K_q^c},t)-\hs(S_p/J_{m,K_p},t)..
\end{align*}
     This completes the proof.
 \end{proof}
 Next, we compute the Hilbert series of the join of a connected graph and a complete graph.
 \begin{lemma}\label{lemma: hs of join of a connected and a complete graph}
     Let $H$ be a connected graph on $[p]$ and $G=H*K_q$, where $q\geq 2$. Then $$\hs(S/J_{m,G},t)=\hs(S_H/J_{m,H},t)+\hs(S/J_{m,K_{p+q}},t)-\hs(S_p/J_{m,K_p},t).$$
 \end{lemma}
 \begin{proof}
     Let $H'=H\sqcup \{w\}$, $G'=H'*K_q^c$, $S_{H'}=S_H[z_{iw}:i\in [m]]$ and $S'=S[z_{iw}:i\in [m]]$. Then, by Theorem \ref{thm: hs of join of two disconnected graphs}, we obtain
     \begin{align*}
\hs(S'/J_{m,G'},t)=&\frac{1}{(1-t)^m}\hs(S_H/J_{m,H},t)+\frac{1}{(1-t)^{mq}}+\hs(S'/J_{m,K_{p+q+1}},t)\\
        &-\hs(S_{H'}/J_{m,K_{p+1}},t)-\hs(S_q/J_{m,K_{q}},t).
    \end{align*}
    Since $w$ is not a free vertex of $G'$, it follows from Theorem \ref{thm: Short exact sequence} that $$J_{m,G'}=J_{m,G'_w}\cap \left ((x_{iw}:i\in [m])+J_{m,G'\setminus w}\right ).$$ Let $Q_1=(x_{iw}:i\in [m])+J_{m,G'\setminus w} $ and $Q_2=J_{m,G'_w}$. Then it is easy to observe that $Q_1+Q_2=(x_{iw}:i\in [m])+J_{m,G'_w\setminus w}$. Note that $G'_w=H(*)^qK_{q+1}$, $G'\setminus w=H*K_q^c$ and $G'_w\setminus w=H*K_q=G$. Consider the short exact sequence 
      \[ 0\rightarrow{\frac{S'}{J_{m,G'}}}\rightarrow \frac{S'}{Q_1}\oplus \frac{S'}{Q_2}\rightarrow \frac{S'}{Q_1+Q_2}\rightarrow 0.\] Then, we have
          $$\hs(S/J_{m,G})=\hs(S'/Q_1,t)+\hs(S'/Q_2,t)-\hs(S'/J_{m,G'},t).$$
Since $\hs(S'/Q_1,t)=\hs(S/J_{m,H*K_q^c},t)$ and $\hs(S'/Q_2,t)=\hs(S'/J_{m,H(*)^qK_{q+1}},t)$, the result follows from Lemma \ref{lemma: hs of join of a graph and some points of a complete graph} and \ref{lemma: hs of join of connected graph and isolated vertices}.          
 \end{proof}
 Now, we are in a position to compute the Hilbert series of the join of arbitrary graphs
 \begin{theorem}\label{thm: hs of join of graphs}
     Let $H$  and $H'$ be two graphs on $[p]$ and $[q]$, respectively. Let $G=H*H'$ be the join of $H$ and $H'$. Then 
     \begin{align*}
        \hs(S/J_{m,H*H'},t)=&\hs(S_H/J_{m,H},t)+\hs(S_{H'}/J_{m,H'},t)\\
        & +\frac{\sum_{i\geq 0}\binom{m-1}{i}\left[\binom{p+q-1}{i}-\binom{p-1}{i}(1-t)^q-\binom{q-1}{i}(1-t)^p\right]t^i}{(1-t)^{m+p+q-1}}.
    \end{align*}
 \end{theorem}
 \begin{proof}
    The case when $H$ and $H'$ both are disconnected is proved in Theorem \ref{thm: hs of join of two disconnected graphs}. Now, assume that $H$ is connected and $H'$ is disconnected. Let $w$ be a new vertex. Set $H''=H\sqcup \{w\}$, $G'=H'*H''$, $S_{H''}=S_H[z_{iw}:i\in [m]]$ and $S'=S[z_{iw}:i\in [m]]$. Let $Q_1=\left ((x_{iw}:i\in [m])+J_{m,G'\setminus w}\right )$ and $Q_2=J_{m,G'_w}$. Since $w$ is not a free vertex of $G'$, it follows from Theorem \ref{thm: Short exact sequence} that $J_{m,G'}=Q_1\cap Q_2$. Note that $Q_1+Q_2=\left ((x_{iw}:i\in [m])+J_{m,G'_w\setminus w}\right )$, $G'_w=H(*)^qK_{q+1}$, $G'_w\setminus w=H*K_q$ and $G'\setminus w=G$. Therefore, it follows from the short exact sequence \[ 0\rightarrow{\frac{S'}{J_{m,G'}}}\rightarrow \frac{S'}{Q_1}\oplus \frac{S'}{Q_2}\rightarrow \frac{S'}{Q_1+Q_2}\rightarrow 0\] that
     $$\hs(S/J_{m,G})=\hs(S'/J_{m,G'},t)+\hs(S/J_{m,H*K_q},t)-\hs(S'/J_{m,H(*)^qK_{q+1}},t).$$
     Note that the graph $G'$ is a join of two disconnected graphs. Therefore, the assertion follows from Theorem \ref{thm: hs of join of two disconnected graphs}, Lemma \ref{lemma: hs of join of a graph and some points of a complete graph}, \ref{lemma: hs of join of a connected and a complete graph}, and Corollary \ref{cor: hs of 2-minors}. 

     Assume now that both $H$ and $H'$ are connected. In view of Lemma \ref{lemma: hs of join of a connected and a complete graph}, we may assume that $H$ is not a complete graph. Let $u$ be a new vertex. Set $H'''=H'\sqcup \{u\}$, $G''=H*H'''$, $S_{H'''}=S_H[z_{iu}:i\in [m]]$ and $S'=S[z_{iu}:i\in [m]]$. Let $Q_1=\left ((x_{iu}:i\in [m])+J_{m,G''\setminus u}\right )$ and $Q_2=J_{m,G''_u}$. Since $u$ is not a free vertex of $G'$, it follows from Theorem \ref{thm: Short exact sequence} that $J_{m,G''}=Q_1\cap Q_2$. Note that $Q_1+Q_2=\left ((x_{iu}:i\in [m])+J_{m,G''_u\setminus u}\right )$, $G''_u=H'(*)^pK_{p+1}$, $G''_u\setminus u=H'*K_p$ and $G''\setminus u=G$. The short exact sequence \[ 0\rightarrow{\frac{S'}{J_{m,G'}}}\rightarrow \frac{S'}{Q_1}\oplus \frac{S'}{Q_2}\rightarrow \frac{S'}{Q_1+Q_2}\rightarrow 0\] gives $$\hs(S/J_{m,G})=\hs(S'/J_{m,G''},t)+\hs(S/J_{m,H'*K_p},t)-\hs(S'/J_{m,H'(*)^pK_{p+1}},t).$$ Since $G''$ is a join of a connected graph $H$ and a disconnected graph $H'''$, by the previous case, we have 
       \begin{align*}
\hs(S'/J_{m,G''})=&\hs(S_H/J_{m,H},t)+\hs(S_{H'''}/J_{m,H'''},t)+\hs(S_{p+q+1}/J_{m,K_{p+q+1}},t)\\
        &-\hs(S_p/J_{m,K_{p}},t)-\hs(S_{q+1}/J_{m,K_{q+1}},t).
    \end{align*}
    Now the result follows from Lemma \ref{lemma: hs of join of a graph and some points of a complete graph}, \ref{lemma: hs of join of a connected and a complete graph} and Corollary \ref{cor: hs of 2-minors}. 
 \end{proof}
 As an immediate consequence, we have the following.
 \begin{corollary}\label{cor: dim of join of graphs}
    Let $H$ and $H'$ be two graphs on $[p]$ and $[q]$, respectively. Then 
    \begin{enumerate}
        \item $\dim(S/J_{m,H*H'})=\max\{\dim(S_H/J_{m,H}),\dim(S_{H'}/J_{m,H'}),m+p+q-1\}$.
        \item $\height(J_{m.H*H'})=\min\{mq+\height(J_{m,H}),mp+\height(J_{m,H'}),(m-1)(p+q-1)\}$.
    \end{enumerate}
\end{corollary} 
Recall that given an integer $r\geq 2$, a graph \(G\) is called an \(r\)-\emph{partite} graph if there exists a partition \(V(G)=V_1\cup\cdots\cup V_r\) such that, for all \(1\leq k\leq r\) and \(i,j\in V_k\), the vertices \(i\) and \(j\) are not adjacent. If, for \(k\neq l\), every vertex of \(V_k\) is adjacent to every vertex of \(V_l\), then we say that \(G\) is a \emph{complete \(r\)-partite} graph, denoted by \(K_{p_1,\ldots,p_r}\), where \(p_k=|V_k|\).
\vspace{2mm}

Now, we compute the Hilbert series of complete $r$-partite graphs, which was computed in \cite{shen-Zhu-2023-GBEI-of-complete-r-partite-graphs}. Our proof uses Theorem \ref{thm: hs of join of graphs}, which provides a simpler and more concise proof. We also compute the dimension and multiplicity of complete $r$-partite graphs.
\begin{corollary}
    Let $G=K_{p_1,\ldots,p_r}$ be the complete $r$-partite graph on the vertex set $[n]=[p_1]\sqcup \cdots\sqcup [p_r]$ with $p_1=p_2=\cdots =p_t>p_{t+1}\geq \cdots\geq p_r$. Then
    $$\hs(S/J_{m,G},t)=\sum_{\substack{j\in[r]}}\left[\frac{1}{(1-t)^{mp_j}}-\frac{\sum_{i\geq 0}\binom{m-1}{i}\binom{p_j-1}{i}t^i}{(1-t)^{m+p_j-1}}\right]+\frac{\sum_{i\geq0}\binom{m-1}{i}\binom{n-1}{i}t^i}{(1-t)^{m+n-1}}.$$ In particular, $\dim(S/J_{m,G})=\max\{mp_1,m+n-1\}$ and 
    $$e_0({S/J_{m,G}})=\begin{cases}
        \binom{m+n-2}{m-1}, \text{ if $mp_1<m+n-1$ ,}\\
       \quad t, \quad\quad\text{ if $mp_1>m+n-1$ ,}\\
        t+\binom{m+n-2}{m-1},\text{ if $mp_1=m+n-1$.}
    \end{cases}$$
\end{corollary}
\begin{proof}
    Note that $G=(\cdots(K_{p_1}^c*K_{p_2}^c)*\cdots)*K_{p_r}^c$. Now the assertion follows from Theorem \ref{thm: hs of join of graphs} applied recursively. Since $p_j<n$ for each $j\in[r]$, we get that $$\dim(S/J_{m,G})=\max\{mp_1,m+n-1\}.$$ Now, the rest follows from the fact that $e_0(S/J_{m,K_n})=\binom{m+n-2}{m-1}$.
\end{proof}
\section{Multiplicity of generalized binomial edge ideals}\label{sec: multiplicity}
In this Section, we focus on the multiplicity of generalized binomial edge ideals. We provide a condition on a vertex that ensures some of the Hilbert coefficients remain unchanged upon its removal. Recall that the minimal primes of $J_{m, G}$ are of the form $P_{T}(m,G)$, where $T\subset[n]$ is a cut set. The height of $P_T(m,G)$ is $$\height(P_{T}(m,G))=(m-1)(n-c(T))+|T|,$$ 
and $$\dim(S/J_{m,G})=\max\left\{(m-1)c(T)+n-|T|\right\}.$$
We now recall the definition of the free clique degree of a vertex.
\begin{definition}\cite[Definition 3.2]{KKP-Hilbert-coefficients-regularity-BEI}
    Let $v \in V(G)$, and let $C$ be a maximal clique containing $v$. We call $C$ a \emph{free clique} of $v$ if either $G = C$, or every vertex of $C$ other than $v$ is a free vertex. The \emph{free clique degree} of $v$, denoted by $\fcdeg_G(v)$, is the number of distinct free cliques of $v$.

    The following result provides a sufficient condition on a vertex such that removing it preserves certain Hilbert coefficients. For \(m=2\), it recovers \cite[Theorem 3.3]{KKP-Hilbert-coefficients-regularity-BEI}. 
\end{definition}
 \begin{theorem}\label{thm: fcdeg condition to remove vertex keeping Hilbert coefficients same}
    Let $G$ a graph and $v\in \iv(G)$ such that $\fcdeg(v)\geq \lceil\frac{t+2}{m-1}\rceil+1$ for some $t\geq 0$. Then $$e_j(S/J_{m,G})=e_j(S/J_{m,G\setminus v}),$$ for all $0\leq j\leq t.$
\end{theorem}
\begin{proof}
   Since $v$ is an internal vertex, it follows from Theorem \ref{thm: Short exact sequence} that $$J_{m,G}=J_{m,{G_v}}\cap \langle\{ x_{iv}:i\in [m]\}+J_{m,{G\setminus v}}\rangle.$$ We claim that $\h(J_{m,G_v})>t+\h(\langle\{ x_{iv}:i\in [m]\}+J_{m,,{G\setminus v}}\rangle).$ Assuming the claim, the result follows from \cite[Lemma 3.1]{KKP-Hilbert-coefficients-regularity-BEI}.
   
   Let $k=\lceil\frac{t+2}{m-1}\rceil+1$ and $\height(J_{m,G_v})=\height(P_{W}(m,G_v))=(m-1)(n-c_{G_v}(W))+|W|$, for some cut set $W$ of $G_v$. Since the induced subgraph on $N_{G_v}[v]$ is a complete graph, it follows that $v\notin W$.  Set $T=W\cup \{v\}$. Then from the proof of \cite[Theorem 3.3]{KKP-Hilbert-coefficients-regularity-BEI}, we have $c_G(T)\geq C_{G_v}(W)+k-1$. 
   Now, 
    \begin{align*}
            \height\left (P_T(m,G)\right ) & = (m-1)(n-c_G(T))+|T| \qquad && \text{}\\
            & = (m-1)(n-c_G(T))+|W|+1.
     \end{align*}
Since $c_{G}(W\cup \{v\})\geq c_{G_v}(W)+k-1$, we have 
     \begin{alignat*}{3}
          \height\left (P_T(m,G)\right )  &\leq (m-1)(n-c_{G_v}(W)-k+1)+|W|+1  \quad && \\
            & = (m-1)(n-c_{G_v}(W))+|W|-(m-1)(k-1)+1\qquad && \\
            & =\height(J_{m,G_v})-(m-1)k+m.
        \end{alignat*}
On the other hand, since $P_T(m,G)=\left\langle x_{iv}:i\in [m] \right \rangle+ P_{T\setminus v}(m,G\setminus v)$, we have $$\height\left (P_T(m,G)\right )= m+ \height\left (P_{T\setminus v}(m,G\setminus v)\right).$$ Observe that $J_{m,G\setminus v}\subset P_{T\setminus v}(G\setminus v)$. Hence,  $\height\left (J_{m,G\setminus{v}}\right ) \leq \height\left (P_{T\setminus v}(m,G\setminus v)\right)$.
Therefore, we get 
$\height(J_{m,G\setminus{v}})\leq \height(P_{T}(m,G))-m=\height(J_{m,G_v}) -(m-1)k.$ This yields
\begin{align*}
    \height(J_{m,G_v})&\geq \height(J_{m,G\setminus{v}})+(m-1)k\\
    &\geq \height(J_{m,G\setminus{v}})+(m-1)\left\lceil\frac{t+2}{m-1}\right\rceil+m-1\\
    &=\h(\langle\{ x_{iv}:i\in [m]\}+J_{m,,{G\setminus v}}\rangle)+(m-1)\left\lceil\frac{t+2}{m-1}\right\rceil-1.
\end{align*}
Note that  $$b\left\lceil\frac{a}{b}\right\rceil=\begin{cases}
    a, \qquad\quad\quad\text{ if $b$ divides $a$},\\
    a+b-r, ~\text{ if $a=bq+r,0<r<b$}.
\end{cases}$$
Therefore, we have $\h(J_{K_m,G_v})>t+\h(\langle\{ x_{iv}:i\in [m]\}+J_{K_m,,{G\setminus v}}\rangle)$. This proves the claim and hence the result.
\end{proof}
As a consequence of Theorem \ref{thm: fcdeg condition to remove vertex keeping Hilbert coefficients same}, we obtain the multiplicity of doubly whiskered graphs. For a graph \(G\), the \emph{doubly whiskered graph} of \(G\), denoted \(W^2(G)\), is the graph obtained by attaching two whiskers to each vertex of \(G\).

\begin{corollary}
    Let \(W^2(G)\) be the doubly whiskered graph of \(G\), and let \(J_{m,W^2(G)}\) be its generalized binomial edge ideal for \(m\geq 3\). Then
$
e_0(S/J_{m,W^2(G)})=1,
$
and, if \(m\geq 4\), then
$
e_j(S/J_{m,W^2(G)})=0
 \text{ for }1\leq j\leq m-3.
$
\end{corollary}
\begin{proof}
    For $t=m-3$, the number $\left\lceil\frac{t+2}{m-1}\right\rceil+1$ is $2$. Notice that every vertex $v\in V(G)$ satisfies $\fcdeg_{W^2(G)}(v)\geq 2$. Therefore, by applying Theorem \ref{thm: fcdeg condition to remove vertex keeping Hilbert coefficients same} recursively, the desired result follows.
\end{proof}
The following example illustrates Theorem \ref{thm: fcdeg condition to remove vertex keeping Hilbert coefficients same}.
\begin{example}
    Let $G$ be the graph shown in Figure \ref{fig:delete of a vertex}. Then $\fcdeg_G(v)=3$. Using Theorem \ref{thm: fcdeg condition to remove vertex keeping Hilbert coefficients same}, we obtain $e_j(S/J_{m,G})=e_j(S/J_{m,G\setminus v})$ for $m\geq 2$ and $0\leq j\leq 2m-4$. Further, notice that $\fcdeg_{G\setminus v}(w)=2$. Therefore, for \(m\geq 3\), applying Theorem \ref{thm: fcdeg condition to remove vertex keeping Hilbert coefficients same} to \(G\setminus v\), we get
$
e_j(S/J_{m,G\setminus v})
=
e_j(S/J_{m,G\setminus\{v,w\}})
$
for \(0\leq j\leq m-3\). In particular, $e_0(S/J_{m,G})=m^3$. 
\end{example}
\begin{figure}[h]
    \centering
\begin{tikzpicture}[scale= 0.7, line cap=round,line join=round]
\draw (0,0)--(-1,1);
\draw (0,0)--(-1,-1);
\draw (-1,-1)--(-1,1);

\draw (0,0)--(1,1);
\draw (1,1)--(2,0);
\draw (2,0)--(1,-1);
\draw (1,-1)--(0,0);

\draw (0,0)--(2,0);
\draw (2,0)--(3.2,0);
\draw (1,1)--(1,-1);

\draw (0,0)--(0,1);

\draw (0,0)--(-0.58,-1.54);
\draw (0,0)--(0.78,-1.5);
\draw (-0.58,-1.54)--(0.78,-1.5);
\fill (0,0) circle (2pt);
\fill (-1,1) circle (2pt);
\fill (-1,-1) circle (2pt);
\fill (1,1) circle (2pt);
\fill (2,0) circle (2pt);
\fill (1,-1) circle (2pt);
\fill (0,1) circle (2pt);
\fill (-0.58,-1.54) circle (2pt);
\fill (0.78,-1.5) circle (2pt);
\fill (3.2,0) circle (2pt);
\node[above right] at (-0.15,0) {$v$};
\node[above right] at (1.8,0) {$w$};
\draw (7,-1)--(7,1);

\draw (9,1)--(10,0);
\draw (10,0)--(9,-1);

\draw (10,0)--(11.2,0);
\draw (9,1)--(9,-1);

\draw (7.42,-1.54)--(8.78,-1.5);
\fill (7,1) circle (2pt);
\fill (7,-1) circle (2pt);
\fill (9,1) circle (2pt);
\fill (10,0) circle (2pt);
\fill (9,-1) circle (2pt);
\fill (7.42,-1.54) circle (2pt);
\fill (8.78,-1.5) circle (2pt);
\fill (11.2,0) circle (2pt);
\fill (8,1) circle (2pt);
\node[above right] at (9.8,0) {$w$};
\node at (0,-2.3) {$G$};
\node at (8.4,-2.3) {$G\setminus v$};
\end{tikzpicture}
    \caption{}\label{fig:delete of a vertex}
\end{figure}
\section{Applications}\label{sec: Applications}
In this Section, we compute the dimension and multiplicity of generalized binomial edge ideals of multifan and wheel graphs. The case \(m=2\) has already been studied in \cite{HS-of-BEI-Arvind-Sarkar-2019}. Thus, throughout this section, we restrict our attention to \(m\geq 3\).
\subsection{Dimension and multiplicity of multifan graph}
\noindent
\vspace{2mm}

Let $P_{n_1},\ldots,P_{n_r}$ be paths on vertices $n_1,\ldots,n_r$ respectively. Then the graph $\{v\}*(\sqcup_{i=1}^rP_{n_i})$ is called a \emph{multifan} graph, denoted by $M_{n_1,\ldots,n_r}$. To determine the dimension and multiplicity of a multifan graph, we first need to establish the dimension and multiplicity for the path graph.

\begin{lemma}\label{lemma: dimension and multiplicity of path graphs}
    Let $G=P_n$ be the path graph on $n$ vertices. Then
    \begin{enumerate}
        \item \label{lemma: dimension of path graphs} $$ \dim(S/J_{m,G})=
    \begin{cases}
        \lceil\frac{n}{2}\rceil m+1, \text{ if $n$ is even},\\
        \lceil\frac{n}{2}\rceil m,\quad\quad\text{ if $n$ is odd}.
    \end{cases}$$
    \item \label{lemma: multiplicity of path graphs} $$ e_0(S/J_{m,G})=
    \begin{cases}
        \dfrac{mn}{2}, \text{ if $n$ is even},\\
       1,\quad\text{  if $n$ is odd}.
    \end{cases}$$
    \end{enumerate}
\end{lemma}
\begin{proof}
\begin{enumerate}
    \item  Let $T\subset[n]$ be a cut set of $G=P_n$. Then observe that $c(T)=|T|+1,$ and hence $\dim(S/J_{m,P_n})=\max\left\{m+n-1+(m-2)|T|: T\in \mathscr{C}(G)\right\}.$ This number is maximum when $T$ is of maximum cardinality. Note that the maximum cardinality of a cut set in $G$ is $\lfloor \frac{n-1}{2}\rfloor$. After simplifying the expression, we obtain the desired result.
    \item  As observed in the proof of part $(1)$, the dimension of $S/J_{m,G}$ is attained for the cut sets of cardinality $\lfloor \frac{n-1}{2}\rfloor$.  Then it follows from \cite[Theorem 7]{GBI-Rauh-2013} and Lemma \ref{lemma: multiplicity of arbitrary ideal} that $$e_0(S/J_{m,G})=\sum_{\substack{T\in \mathscr{C}(P_n),\\|T|=\lfloor \frac{n-1}{2}\rfloor}}e_0(P_T(m,G)).$$
 If $n$ is odd, then $T=\{2,4,\ldots,n-1\}$ is the cut set of cardinality $\lfloor \frac{n-1}{2}\rfloor$. Note that $P_{T}(m,G)$ is generated by variables. Therefore, we have $e_0(S/J_{m,G})=1$. 

Now, suppose $n$ is even. In this case, for every cut set $T$ of cardinality $\lfloor \frac{n-1}{2}\rfloor$, there is a unique edge from $\{\{2i-1,2i\}:1\leq i\leq\frac{n}{2} \}$ as a component of $P_n\setminus T$, and conversely. Therefore, there are $\frac{n}{2}$ cut sets of cardinality $\lfloor \frac{n-1}{2}\rfloor$. Note that the Hilbert series of $P_T(m,G)$ with $|T|=\lfloor \frac{n-1}{2}\rfloor$ is $\frac{1+(m-1)t}{(1-t)^{\frac{mn}{2}+1}}$. Therefore, we have $e_0(S/J_{m,G})=\frac{mn}{2}$. This completes the proof.
\end{enumerate}
\end{proof}

    We are now ready to compute the multiplicity of multifan graphs.

\begin{theorem}\label{thm: hs of multifan graph}
    Let $G=M_{n_1,\ldots,n_r}$, $r\geq 2$ be a multifan graph and $n=\sum_{i=1}^rn_i$. Then $$\dim(S/J_{m,G})=\sum_{i=1}^r\dim(S/J_{m,P_{n_i}})\text{ and }e_0(S/J_{m,G})=\prod_{n_i=\text{ even }}\dfrac{mn_i}{2}.$$
\end{theorem}
\begin{proof}
It follows from Corollary \ref{cor: dim of join of graphs} that $\dim(S/J_{m,G})=\max\{\dim(S/J_{m,\sqcup_{i=1}^rP_{n_i}}),m+n\}$. Let $d=\dim(S/J_{m,\sqcup_{i=1}^rP_{n_i}})$. Then, we have $d=\sum_{i=1}^r\dim(S/J_{m,P_{n_i}})$.
Let $A=\{n_i:n_i\text{ is even }\}$, $B=\{n_i:n_i\text{ is odd }\}$. Using Lemma \ref{lemma: dimension and multiplicity of path graphs}\eqref{lemma: dimension of path graphs}, we obtain $$d=\sum_{n_i\in A}(\frac{mn_i}{2}+1)+\sum_{n_i\in B}\frac{m(n_i+1)}{2}.$$ Now, we show that $d-m-n>0$. Note that $d-m-n=\frac{1}{2}(m(n-2)-2n+2|A|+m|B|)$. Since $m\geq 3$ and $n\geq 2|A|+|B|$, we get $d-m-n\geq 4|A|+(m+1)|B|-6>0$. Thus, we obtain $\dim(S/J_{m,G})=\sum_{i=1}^r\dim(S/J_{m,P_{n_i}})$. 

Next, we compute the multiplicity. Using Theorem \ref{thm: hs of join of graphs}, we get
   \begin{align*}
      \hs(S/J_{m,G})&=\hs(S/J_{m,\sqcup_{i=1}^rP_{n_i}},t)+\frac{\sum_{i\geq 0}\binom{m-1}{i}\left[\binom{n}{i}-(1-t)\binom{n-1}{i}\right]t^i}{(1-t)^{m+n}}\\
      &=\frac{Q(t)}{(1-t)^{d}}+\frac{\sum_{i\geq 0}\binom{m-1}{i}\left[\binom{n}{i}-(1-t)\binom{n-1}{i}\right]t^i}{(1-t)^{m+n}},
   \end{align*}
    where $Q(t)\in \mathbb Z[t]$. Since $d>m+n$, it follows that  $$e_0(S/J_{m,G})=e_0(S/J_{m,\sqcup_{i=1}^rP_{n_i}})=\prod_{i=1}^r e_0(S/J_{m,P_{n_i}}).$$ By Lemma \ref{lemma: dimension and multiplicity of path graphs}\eqref{lemma: multiplicity of path graphs}, we obtain $e_0(S/J_{m,G})=\prod_{\substack{n_i=\text{ even}}} \dfrac{mn_i}{2}$. This completes the proof.  
\end{proof}
The following remark shows that the condition
$
\fcdeg_G(v)\geq \left\lceil\frac{t+2}{m-1}\right\rceil+1
$ in Theorem \ref{thm: fcdeg condition to remove vertex keeping Hilbert coefficients same} is sufficient but not necessary for removing a vertex.

\begin{remark}
    Let $G=(P_{n_1}\sqcup P_{n_2}\sqcup\cdots\sqcup P_{n_r})*\{v\}$, $r\geq 2,n_i\geq 3$ for $1\leq i\leq r$. Then, by Theorem \ref{thm: hs of multifan graph}, we have $e_0(S/J_{m,G})=\prod_{i=1}^r e_0(S/J_{m,P_{n_i}})=e_0(S/J_{m,G\setminus v})$. However, the vertex $v$ has $\fcdeg_G(v)=0.$ 
\end{remark}
 \subsection{Dimension and multiplicity of wheel graph}
\noindent
\vspace{2mm}

 Let $C_n$ denote the cycle on $n$ vertices for $n\geq 3$.  A \emph{wheel} graph on $n+1$ vertices, denoted $W_{n+1}$, is the join of $C_n$ and $\{v\}$, i.e., $W_{n+1}=C_n*\{v\}$. Note that for $n=3$, the cycle $C_3$ and $C_3*\{v\}$ are complete graphs. Dimension and multiplicity are known in this case by Corollary \ref{cor: hs of 2-minors}. Therefore, we assume $n\geq 4$. In order to compute the dimension and multiplicity of $W_{n+1}$, we first compute the same for $C_n$.
 \begin{theorem}\label{thm: dim and multiplicity of Cycles}
     Let $G=C_n$ be the cycle on $n\geq 4$ vertices. Then
     \begin{enumerate}
         \item\label{dim of cycle} $$\dim(S/J_{m,G})=\begin{cases}
             \dfrac{mn}{2}, \qquad\qquad\text{ if $n$ is even},\\
             \dfrac{m(n-1)}{2}+1, \text{ if $n$ is odd}.
         \end{cases}$$
         \item\label{multiplicity of cycle} $e_0(S/J_{m,G})=\begin{cases}
             12, \text{ if $(m,n)=(3,4)$},\\
             30, \text{ if $(m,n)=(3,5)$},\\
             2, ~ \text{ if $n$ is even and $(m,n)\neq (3,4)$},\\
             mn, \text{ if $n$ is odd and $(m,n)\neq (3,5)$}.
         \end{cases}$
     \end{enumerate}
 \end{theorem}
 \begin{proof}
     By \cite[Theorem 7]{GBI-Rauh-2013}, the minimal primes of $J_{m,G}$ are of the form $P_T(m,G)$, where $T\in \mathscr{C}(G)$. The height of $P_T(m,G)$ is given by $\height(P_T(m,G))=(m-1)(n-c_T(G))+|T|.$ 
     \begin{enumerate}
         \item 
     Observe that $c_T(G)=|T|$ for all $\emptyset \neq T\in \mathscr{C}(G)$. Therefore, we have $\height(P_{T}(m,G))=(m-1)n-|T|(m-2)$ for $\emptyset \neq T\in \mathscr{C}(G)$. Note that the maximum cardinality of a cut set of \(G\) is \(\left\lfloor\frac{n}{2}\right\rfloor\). Hence, the minimal height among $P_T(m,G)$, $T\neq \emptyset$ is $(m-1)n-\left\lfloor\frac{n}{2}\right\rfloor(m-2)$. Since $n\geq 4$ and $\height(P_{\emptyset}(m,G))=(m-1)(n-1)$, we get $(m-1)n-\left\lfloor\frac{n}{2}\right\rfloor(m-2)\leq \height(P_{\emptyset}(m,G))$. Therefore, we obtain $$\dim(S/J_{m,G})=mn-(m-1)n+\left\lfloor\frac{n}{2}\right\rfloor(m-2)=n+\left\lfloor\frac{n}{2}\right\rfloor(m-2).$$ Simplifying the above expression, we get the desired result.
     \item From the proof of $(1)$, we have $(m-1)n-\left\lfloor\frac{n}{2}\right\rfloor(m-2)\leq \height(P_{\emptyset}(m,G)).$ Now, suppose $(m-1)n-\left\lfloor\frac{n}{2}\right\rfloor(m-2)= \height(P_{\emptyset}(m,G))=(m-1)(n-1)$. Then, we have $\lfloor\frac{n}{2}\rfloor=1+\frac{1}{m-2}$. This holds if and only if $m=3$ and $ n=4$ or $5$. If \((m,n)=(3,4)\), then the primes \(P_{\emptyset}(m,G)\), \(P_{T_1}(m,G)\), and \(P_{T_2}(m,G)\), where \(T_1,T_2\) are cut sets of cardinality \(2\), all have dimension equal to \(\dim(S/J_{m,G})\).
 By Lemma \ref{lemma: multiplicity of arbitrary ideal}, and the fact that $e_0(S/P_{T_i}(m,G))=1$ for $i=1,2$ and $e_0(S/P_{\emptyset}(m,G))=10$, we get $e_0(S/J_{m,G})=12$. Now, assume $(m,n)=(3,5)$. In this case, there are five cut sets of cardinality $2$ such that the corresponding prime has multiplicity $3$. Therefore, by Lemma \ref{lemma: multiplicity of arbitrary ideal} and Corollary \ref{cor: hs of 2-minors}, we obtain $e_0(S/J_{m,G})=30$. 

     Now, suppose $(m,n)\neq (3,4),(3,5)$. Then, we get $\dim(S/J_{m,G})=\dim(S/P_T(m,G))$ for all cut sets $T$ of with $|T|=\lfloor\frac{n}{2}\rfloor$. Therefore, by Lemma \ref{lemma: multiplicity of arbitrary ideal}, we have $$e_0(S/J_{m,G})=\sum_{|T|=\lfloor\frac{n}{2}\rfloor}e_0(S/P_T(m,G)).$$
     If $n$ is even, then there are two cut sets of cardinality $\lfloor\frac{n}{2}\rfloor$. The corresponding primes are generated by variables. Therefore, we get $e_0(S/J_{m,G})=2$. If $n$ is odd, then there are $n$ cut sets of cardinality $\lfloor\frac{n}{2}\rfloor$. Note that each of the corresponding primes is generated by variables and $J_{m,K_2}$. Therefore, by Corollary \ref{cor: hs of 2-minors}, we have $e_0(S/P_T(m,G))=m$ for all $T\in \mathscr{C}(G)$ with $|T|=\lfloor\frac{n}{2}\rfloor$. Hence, the result follows.  
     \end{enumerate} 
 \end{proof}
Now, we compute the dimension and multiplicity of wheel graphs.
\begin{theorem}
    Let $W_{n+1}=C_n*\{v\}$ be the wheel graph. Then 
    \begin{enumerate}
        \item $$\dim(S/J_{m,W_{n+1}})=\begin{cases}
            7,\quad \qquad\qquad\text{ if $(m,n)=(3,4)$},\\
            8, \quad\qquad\qquad\text{ if $(m,n)=(3,5)$},\\
            \dfrac{mn}{2},\quad \quad \qquad \text{ if $n$ is even and $(m,n)\neq (3,4)$},\\
            \dfrac{m(n-1)}{2}+1, \text{ if $n$ is odd and $(m,n)\neq (3,5)$}.
        \end{cases}$$
        \item Let $A=\{(3,4),(3,5),(3,6),(3,7),(4,4),(4,5)\}$. Then $$e_0(S/J_{m,W_{n+1}})=
        \begin{cases}
            15, &\text{ if $(m,n)=(3,4)$},\\
            21, &\text{ if $(m,n)=(3,5)$},\\
            30, &\text{ if $(m,n)=(3,6)$},\\
            57, &\text{ if $(m,n)=(3,7)$},\\
            37, &\text{ if $(m,n)=(4,4)$},\\
            76, &\text{ if $(m,n)=(4,5)$},\\
            2, &\text{ if $n$ is even and $(m,n)\notin A$},\\
            mn, &\text{ if $n$ is odd and $(m,n)\notin A$}.
        \end{cases}$$
    \end{enumerate}
\end{theorem}
\begin{proof} 
    \begin{enumerate}
        \item By Corollary \ref{cor: dim of join of graphs}, we have $\dim(S/J_{m,W_{n+1}})=\max\{\dim(S/J_{m,C_n}),m+n\}$. Now, it is easy to check that $\dim(S/J_{m,C_n})\geq m+n$ for all $m$ and $n$ with $(m,n)\neq (3,4) ,(3,5)$. Hence, the statement follows from Theorem \ref{thm: dim and multiplicity of Cycles}\eqref{dim of cycle}.
        \item  It follows from Theorem \ref{thm: hs of join of graphs} that \begin{equation}\label{Eq: Hilbert series of wheel graph}
        \hs(S/J_{m,W_{n+1}})=\frac{Q_1(t)}{(1-t)^d}+\frac{Q_2(t)}{(1-t)^{m+n}}-\frac{Q_3(t)}{(1-t)^{m+n-1}},
    \end{equation}
    where $Q_i(t)\in \mathbb Z[t]$ with $Q_i(1)\neq 0$ for $i=1,2,3$ and $d=\dim(S/J_{m,C_n})$. Therefore, we have $$e_0(S/J_{m,W_{n+1}})=\begin{cases}
            e_0(S/J_{m,C_n}),\qquad\qquad \qquad\qquad\text{ if $d>m+n$},\\
            e_0(S/J_{m,C_n})+e_0(S/J_{m,K_{n+1}}),~\text{ if $d=m+n$},\\
            e_0(S/J_{m,K_{n+1}}), \quad\qquad\qquad\qquad\text{ if $d<m+n$}.
        \end{cases}$$ 
        From the proof of part $(1)$, we have $d<m+n$ if $(m,n)=(3,4)$ and $(3,5)$. Now, we solve $d=m+n$. Using Theorem \ref{thm: dim and multiplicity of Cycles}, we have 
        $$\begin{cases}
         mn=2m+2n,& \text{ if $n$ is even, }\\
         m(n-1)+2=2m+2n&  \text{ if $n$ is odd}.
        \end{cases}$$
Solving these, we have $(m,n)=(4,4),(3,6),(3,7),(4,5)$. Now, using the fact that $e_0(S/J_{m,K_n})=\binom{m+n-2}{m-1}$ and Theorem \ref{thm: dim and multiplicity of Cycles}\eqref{multiplicity of cycle}, we get the desired result. 
    \end{enumerate}
\end{proof}
\section*{Acknowledgment}
 The author would like to thank Rajiv Kumar for his helpful suggestions. The author also acknowledges the University Grants Commission, Government of India, for the financial support.

\bibliographystyle{plain}
	\bibliography{bib}
\end{document}